\documentclass[preprint, 12pt]{elsarticle}%
\usepackage{bbm}
\usepackage{mathrsfs} 
\usepackage{amsmath}
\usepackage{amsfonts}
\usepackage{amssymb}
\usepackage{graphicx}%
\providecommand{\U}[1]{\protect\rule{.1in}{.1in}}
\newtheorem{theorem}{Theorem}

\newtheorem{claim}[theorem]{Claim}

\newtheorem{definition}[theorem]{Definition}

\newtheorem{lemma}[theorem]{Lemma}

\newtheorem{remark}[theorem]{Remark}

\newenvironment{proof}[1][Proof]{\noindent\textbf{#1.} }{\ \rule{0.5em}{0.5em}}

\newcommand{\R} {\ensuremath {\mathbb{R}}}
\newcommand{\E} {\ensuremath {\mathbb{E}}}

\newcommand{\calC} {\ensuremath {\mathcal{C}}}
\newcommand{\scrF} {\ensuremath {\mathscr{F}}}
\newcommand{\dd}  {\ensuremath{\mathrm{d}}}

\journal{JFA}

\begin{document}
\begin{frontmatter}





\title{Noise Prevents Singularities in Linear Transport Equations}
\author[a]{E. Fedrizzi\corref{corr1}}  
\ead{fedrizzi@math.univ-paris-diderot.fr}
\cortext[corr1]{Corresponding author}

\author[b]{F. Flandoli}
\ead{flandoli@dma.unipi.it}

\address[a]{Laboratoire de Probabilit\'es et Mod\`eles Al\'eatoires, Universit\'{e} Paris Did\'erot, Sorbonne Paris Cit\'e, 75205 Paris, France.}
\address[b]{Dipartimento di Matematica Applicata, Universit\`{a} di Pisa, Italia.  }

\begin{abstract}
A stochastic linear transport equation with multiplicative noise
is considered and the question of no-blow-up is investigated. The drift is
assumed only integrable to a certain power. Opposite to the deterministic
case where smooth initial conditions may develop discontinuities, we prove
that a certain Sobolev degree of regularity is maintained, which implies
H\"older continuity of solutions. The proof is based on a careful analysis
of the associated stochastic flow of characteristics.
\end{abstract}


\begin{keyword}
stochastic linear trasport equation \sep no-blow-up \sep integrable drift \sep Sobolev initial conditions 
\MSC[2010] 35R60 
\sep 60H15 
\sep 35R05  
\sep 60H30  
\end{keyword}
\end{frontmatter}

\section{Introduction}

Consider the stochastic linear transport equation in Stratonovich form%
\[
\frac{\partial u}{\partial t}+b\cdot\nabla u+\sigma\nabla u\circ\frac{\dd W}%
{\dd t}=0,\qquad u|_{t=0}=u_{0}.
\]
Here $W=\left(  W_{t}\right)  _{t\geq0}$ is a $d$-dimensional Brownian motion
defined on a filtered probability space $\left(  \Omega,\scrF,\scrF_{t},P\right)  $,
the drift $b:\left[  0,T\right] \times \mathbb{R}^{d} \rightarrow\mathbb{R}%
^{d}$ is a given deterministic vector field, $\sigma\in\mathbb{R}$ and
$u_{0}:\mathbb{R}^{d}\rightarrow\mathbb{R}$ are given and the solution
$u=u\left(  x,t\right)  $ will be a scalar random field on $\left(
\Omega,\scrF,\scrF_{t},P\right)  $ defined for $\left( t, x\right)  \in \left[  0,T\right] \times  \mathbb{R}%
^{d}  $.

We deal with the problem of singularities of $u$ starting from a regular
initial condition $u_{0}$. When $\sigma=0$ and $b$ is not Lipschitz,
singularities may appear, in the form of discontinuities (or
blow-up of derivatives), as in the simple example $d=1$, $b\left(  x\right)
=-sign\left(  x\right)  \sqrt{\left\vert x\right\vert }$: any non-symmetric
smooth initial condition $u_{0}$ develops a discontinuity at $x=0$ for any
$t>0$, because there are different, symmetric, initial conditions $x_{0}$ for
the associated equation of characteristics%
\[
x^{\prime}\left(  t\right)  =b\left( t, x\left(  t\right) \right)  ,\quad
x\left(  0\right)  =x_{0}%
\]
which coalesce at $x=0$ at any arbitrary positive time. Opposite to the
question of uniqueness of weak $L^{\infty}$ solutions, where positive results
have been given under relatively weak assumptions on $b$, see for instance
\cite{DPL} and \cite{Amb}, it seems that good results of no blow-up are not available in the
deterministic case when $b$ is not Lipschitz.

The purpose of this paper is to show that, for $\sigma \neq0$ and $b$ of
class%
\[
b\in L_p^q := L^{q}\left(  0,T;L^{p}(\mathbb{R}^{d},\mathbb{R}^{d})\right)\, ,
\]
\begin{equation}
p,q\geq2 \, ,  \qquad    \frac{d}{p}+\frac{2}{q}<1    \label{pq}
\end{equation}
some regularity of the initial condition is maintained, in particular
discontinuities do not appear. We prove the following result.

\begin{theorem}
\label{Thm intro}If $\sigma\neq0$, (\ref{pq}) holds and $u_{0}\in
\cap_{r\geq1}W^{1,r}\left(  \mathbb{R}^{d}\right)  $ then there exists a unique solution $u$
such that
\[
P\left(  u\left(  t,\cdot\right)  \in\cap_{r\geq1}W_{loc}^{1,r}\left(
\mathbb{R}^{d}\right)  \right)  =1
\]
 for every $t\in\left[  0,T\right]  $.

\end{theorem}

The unique solution in this class is given by a representation
formula, in terms of $u_{0}$, involving a weakly differentiable stochastic
flow. By Sobolev embedding theorem, $u(t,\cdot)$ is $\alpha$-H\"older continuous for
every $\alpha\in (0,1)$, with probability one. Hence, from smooth initial
conditions, discontinuities cannot arise.

The precise formulation of the concept of solution and other details are
given in the sequel.

The intuitive idea is that, opposite to the deterministic case, when
$\sigma\neq0$ the characteristics cannot meet. They satisfy the
stochastic equation
\begin{equation}\label{SDE}
\dd X_{t}=b\left(  X_{t},t\right)  \dd t+\sigma \dd W_{t}
\end{equation}
which generates, under assumptions (\ref{pq}), a stochastic flow of H\"{o}lder
continuous homeomorphisms, with some weak form of differentiability. The
existence of an H\"{o}lder continuous stochastic flow has been proved in
\cite{FeFl1}, \cite{FeFl2}, \cite{Z}. A differentiability property in terms of finite increments
has been given in \cite{FeFl2}. Here we establish Sobolev type differentiability. A
similar Sobolev regularity of the flow is investigated in \cite{MNP} by different tools (Malliavin calculus). See also \cite{MP}.

The assumption (\ref{pq}) has been introduced in the framework of stochastic
differential equations by \cite{KR} who have proved strong uniqueness. In the fluid
dynamic literature, with $\leq$ in place of $<$, it is known as the
Ladyzhenskaya-Prodi-Serrin condition. One of its main consequences is that it
gives uniform bounds on gradients of solutions to an auxiliary parabolic
problem (see Theorem \ref{Main PDE Theorem} below) essential for our approach, along with good properties of
the second derivatives. 

The possibility that noise may prevent the emergence of singularities is an
intriguing phenomenon that is under investigation for several systems. For
linear transport equations with $b\in L^{\infty}\left(  \left[  0,T\right]
;C_{b}^{\alpha}\left(  \mathbb{R}^{d},\mathbb{R}^{d}\right)  \right)  $ it may
be deduced from \cite{FGP1} (the result presented here is more general). For
nonlinear systems there are negative results, like the fact that noise does
not prevent shocks in Burgers equation, see \cite{F}, and positive results for
special kind of singularities (collapse of measure valued solutions) for the
vorticity field of 2D Euler equations, see \cite{FGP2}, and for 1D Vlasov-Poisson
equation, see \cite{DFV}. Moreover, for Schr\"{o}dinger equations, there are
several theoretical and numerical results of great interest, see
 \cite{BD1} - 
 \cite{DT}. 
We do not list here the results concerning the restored uniqueness due to noise
and address to the lecture note \cite{F} on this subject.

After the result of Theorem \ref{Thm intro}, it remains open the question
whether the solution is Lipschitz continuous (or more)\ when $u_{0}\in
W^{1,\infty}\left(  \mathbb{R}^{d}\right)  $ (or more). In dimension $d$ we
think that this is a difficult question under assumption (\ref{pq}). The
answer is positive when $b\in L^{\infty}\left(  \left[  0,T\right]
;C_{b}^{\alpha}\left(  \mathbb{R}^{d},\mathbb{R}^{d}\right)  \right)  $
because the stochastic flow is made of diffeomorphisms, see \cite{FGP1} and it is
also positive in dimension $d=1$ for certain discontinuous drift $b$,
including for instance $b\left(  x\right)  =sign\left(  x\right)  $, see \cite{Att}.

It must be emphasized that, although this \textquotedblleft regularization by
noise\textquotedblright\ may look related to the regularization produced by
the addition of a Laplacian to the equation, in fact it preserves the
hyperbolic structure of the equation. The equations remain reversible and the
solution at time $t$ is, in the problem treated in this paper, just given by
the initial condition composed with a flow. If the initial condition has a
discontinuity, the solution also has a discontinuity; no smoothing effect is
introduced. However, the emergence of singularities (shocks in our case) is
prevented. 

The work is organized as follows. In Section \ref{sec 2} we present some results on regularity and approximation properties of the flow associated to the SDE (\ref{SDE}). They are obtained via the study of an associated SDE and the regularity of its solutions. The main results are contained in Lemmas \ref{lemma convergenza SDE} and \ref{lemma stima derivata}, while more technical results are collected in the Appendix (Section \ref{sec app1}). In Section \ref{sec existence} we define weakly differentiable solutions of the SPDE and prove their existence in Theorem \ref{main teo}. A technical result on convergence  of random fields in Sobolev spaces is left to the last Appendix. Finally, uniqueness of weakly differentiable solutions of the SPDE is proved in Section \ref{sec 1!}.

\section{Convergence Results}\label{sec 2}

In this section we present some technical results on an associated SDE that we will study as an intermediate step to obtain some regularity and approximation properties of the flow associated to the SDE (\ref{SDE}). The main results are contained in Lemmas \ref{lemma convergenza SDE} and \ref{lemma stima derivata}.

Let us start by setting the notation used and recalling some results. We will use the following auxiliary SDE, introduced in \cite{FeFl2}:
\begin{equation}\label{new SDE}
\dd Y_t = \lambda U(t,\gamma_0^{-1}(x) ) \dd t + \big(\nabla U(t,\gamma_0^{-1}(x) ) + Id \big) \dd W_t \ , \qquad Y_0=x \ .
\end{equation}
The link between this SDE and the one presented in the introduction is given by the $\calC^1$-diffeomorphism $\gamma_t$: $Y_t = \gamma_t \circ X_t \circ \gamma_0^{-1} $, where $\gamma_t(x) = x+ U(t,x)$. Here, $U:\R^{d+1}\to\R^d$ is the solution of the PDE
\begin{equation*}
 \left\{ \begin{array} [c]{c}
\partial_t U+\frac{1}{2}\Delta U + b\cdot \nabla U - \lambda U + b =0\\
U(T,x)=0
\end{array} . \right.
\end{equation*}
This PDE is well posed in the space $$H_{2,p}^q(T):=L^q \, \big(0,T\,;\,W^{2,p}(R^d)\big) \cap W^{1,q}  \, \big(0,T\,;\,L^p(R^d)\big)\, ;$$ we report here the precise result, given by \cite[Theorem 3.3]{FeFl2}. 

\begin{theorem}\label{Main PDE Theorem} Take $p,q$ such that (\ref{pq}) holds, $\lambda>0$ and two vector fields $b, f\,(t,x):\R^{d+1}\rightarrow\R^d$ belonging to $L_p^q(T)$. Then in $H_{2,p}^q(T)$ there exists a unique solution of the backward parabolic system
\begin{equation}  \label{PDE}
 \left\{ \begin{array} [c]{c}
\partial_t u+\frac{1}{2}\Delta u + b\cdot \nabla u - \lambda u + f =0\\
u(T,x)=0
\end{array} . \right.
\end{equation}
For this solution there exists a finite constant $N$ depending only on $d,p,q,T,\lambda$ and $\|b\|_{L_p^q(T)}$ such that
\begin{equation} \label{Stima di Schauder}
\|u\|_{H_{2,p}^q(T)} \le N \|f\|_{L_p^q(T)}.
\end{equation}
\end{theorem}

We will use the result of this theorem with $f=b$.

Let $b^n$ be a sequence of smooth vector fields converging to $b$ in $L_p^q$. Let $U$ be the unique solution to the PDE (\ref{PDE}) provided by the above Theorem and $U^n$ the solutions obtained using the approximating vector fields $b^n$. Lemma \ref{lemma U^n} shows that the vector fields $U^n$ converge in $H_{2,p}^q$ to $U$.

In \cite{FeFl2} is also proved the existence of H\"older flows of homeomorphisms for the two SDEs above, which we denote by $\phi_t(\cdot)$ for the SDE (\ref{SDE}), and $\psi_t(\cdot)$ for (\ref{new SDE}). We will use $\phi_{t}^{n}(\cdot)$ to denote the flows obtained for the approximating vector fields $b^n$, and $\psi_{t}^{n}(\cdot) $ for the flows corresponding to the auxiliaries SDEs obtained via the diffeomorphisms $\gamma^n_t = Id + U^n(t,\cdot)$. We will use $\phi_{0}^{t,n}(\cdot)$, and $\psi_{0}^{t,n}(\cdot)$ for the inverse flows.\\
We can now state and prove the two main regularity results on the flows $\phi_{0}^{t,n}$.


\begin{lemma}\label{lemma convergenza SDE}      
For every $R>0$, $p\geq1$ and $x,y\in B_R$,
\begin{equation*}
\lim_{n\rightarrow\infty} \sup_{t\in\left[  0,T\right]  }\sup_{x\in B_R  } E\Big[ \big\vert \phi_{0}^{t,n}(x)-\phi_{0}^{t}(y)\big\vert ^{p} \Big]  \leq C_{p,T} \left\vert x-y\right\vert ^{p} \ .
\end{equation*}
In particular,
\begin{equation}\label{property 1}
\lim_{n\rightarrow\infty}\sup_{t\in\left[  0,T\right]  }\sup_{x\in B_R }E\left[  \left\vert \phi_{0}^{t,n}\left(  x\right)  -\phi_{0}%
^{t}\left(  x\right)  \right\vert ^{p}\right]  =0.
\end{equation}
\end{lemma}

\begin{proof}
\textbf{Step 1} (preliminary estimates). 
By lemma \ref{lemma U^n} we have that for every $r>0$ there exist a function $f$ s.t. $\lim_{n\to\infty} f(n)=0$ and
\begin{align*}
\sup_{x\in\R^d}\sup_{t\in[0,T]} \big|\nabla U(t,x)\big|  &\le \frac{1}{2} ,\\
\sup_{x\in B_r}\sup_{t\in [0,T]} \big|U^n(t,x)-U(t,x)\big|  &\le f(n)\ ,\\
\sup_{x\in B_r}\sup_{t\in [0,T]} \big|\nabla U^n(t,x)-\nabla U(t,x)\big|  &\le f(n)\ .
\end{align*}
Since $\phi_t(x)$ is jointly continuous in space and time, there exist an $r<\infty$ s.t. the image of $B_R\times [0,T]$ will be contained in $B_r$ for all $t\le T$. In the following we will always take $x,y\in B_R$. It follows that 
\begin{align*}
\big| U^n(t,\phi_t^{n}(x)) - U(t,\phi_t(y))\big|   &\le f(n) + \frac{1}{2} \big| \phi^{n}_t(x) - \phi_t (y)\big| \ , \\
\big| \nabla U^n(t,\phi_t^{n}(x)) - \nabla U(t,\phi_t(y))\big|   &\le f(n) +  \big| \nabla U^n(t,\phi_t^{n}(x)) - \nabla U^n(t,\phi_t(y))\big| \ .
\end{align*}
To shorten notation, we will write $\phi^n$ and $\phi$ to denote $\phi^{n}_t(x)$ and $\phi_t(y)$, $U^n(\phi^n)$ and $U^n(x)$ to denote $U^n(t,\phi^n)$ e $U^n(0,x)$, etc. The same holds for the flows of the SDE (\ref{new SDE}). 
From the definition $\psi_t^n = \gamma_t \circ \phi_t^n \circ (\gamma_0^n)^{-1} $ and the properties of the diffeomorphisms $\gamma_t^n$ obtained from Lemma \ref{lemma U^n} and Remark \ref{uniformity in n}, we immediately have
\begin{align}
|\psi^n - \psi| &\ge |\phi^n - \phi + U^n(\phi^n) - U^n(\phi) | - f(n) \nonumber  \\
&\ge \frac{1}{2} |\phi^n - \phi| - f(n) \label{X-Y}  \\
2 & \Big( |\psi^n - \psi| +  f(n) \Big)  \ge |\phi^n - \phi|  \nonumber
\end{align}
and
\begin{align}\label{X>Y}
|\psi^n - \psi| \le \frac{3}{2} |\phi^{n} - \phi| + f(n) \ .
\end{align}\\

\textbf{Step 2}  (computations). 
We start by proving the convergence of the flows of the auxiliary SDE (\ref{new SDE}). By It\^o formula, for any $a\ge2$
\begin{align*}
\frac{1}{a} \dd \Big| \psi^n-\psi\Big| ^a =&\ \Big| \psi^n-\psi\Big| ^{a-2} \bigg\{\lambda \big\langle (\psi^n-\psi) , U^n(\phi^n)-U(\phi) \big\rangle_{\R^d} \dd t\\
  & \ \ +  \big\langle (\psi^n-\psi) , \big(\nabla U^n(\phi^n)- \nabla U(\phi)\big)\cdot \dd W_t \big\rangle_{\R^d}\\
  & \ \ +  \frac{a-1}{2} \mathrm{Tr} \bigg( \big[ \nabla U^n(\phi^n)- \nabla U(\phi)\big] \big[ \nabla U^n(\phi^n)- \nabla U(\phi)\big]^T \bigg) \dd t   \bigg\}     \\
  =&\ \Big| \psi^n-\psi\Big| ^{a-2} \bigg\{ A_{1} + A_{2} + A_{3} \bigg\} \ .
\end{align*}
Let us analyze the three terms $A_1,A_2,A_3$. Using (\ref{X-Y}) we have
\begin{align*}
A_1 &\le \lambda |\psi^n-\psi| \Big( f(n) + \frac{1}{2} |\phi^n-\phi| \Big) \, \dd t \\
&\le \lambda  |\psi^n-\psi|^2 \dd t + 2 \lambda f(n)  |\psi^n-\psi| \, \dd t \ .
\end{align*}
Since $\nabla U^n$ is bounded (uniformly in $n$, see Lemma \ref{lemma U^n}) and by (\ref{integrability Y}) $|\psi^n|^a$ belongs to $L^2(\Omega\times[0,T])$ for any $a\ge1$, we can write $A_2= \dd M^n_t$, where for every $n$, $\dd M^n_t$ is the differential of a zero mean martingale. As for the third term, using twice the inequality $(\alpha+\beta)^2\le 2(\alpha^2+\beta^2)$ and the estimates of the first step, we get
\begin{align*}
\frac{2}{d^2(a-1)} A_3 & \le \Big\{ 2 \big| \nabla U^n(\phi^n) - \nabla U^n(\phi) \big|^2  + 2 f^2(n)\Big\} \, \dd t \\
&\le  \big| \phi^n - \phi \big|^2 \dd A_t^{n} + 2 f^2 (n) \, \dd t \\
&\le  8  \big| \psi^n - \psi \big|^2 \dd A_t^n + 8 f^2(n) \, \dd A_t^n  + 2 f^2 (n) \, \dd t \ ,
\end{align*}  
where for every $n$
\begin{equation}\label{eq A^n}
A_t^n := 2 \int_0^t \frac{\big| \nabla U^n(\phi^n_s) - \nabla U^n(\phi_s) \big |^2}{\big| \phi^n_s - \phi_s \big|^2} \mathbbm{1}_{\{\phi^n_s\neq \phi_s\}} \dd s
\end{equation}
is a nondecreasing adapted stochastic process, with $A_0^n=0$, and uniformly in $n$ $\E[A_T^n]\le C<\infty$, see Lemma \ref{lem e^A^n}. Set $B_t^n:= \big[4\, d^2 a (a-1)\big] A_t^n $.
From the above estimates and after renaming $M_t$ (which remains a zero mean martingale
), we get
\begin{align*}
\dd \Big( e^{-B_t^n} \big| \psi^n-\psi \big|^a \Big) \le &\  e^{-B_t^n} \bigg[ a  \lambda |\psi^n-\psi|^a    + 2a\lambda f(n) |\psi^n-\psi|^{a-1}   \bigg] \dd t \\
& + \dd M_t  + f^2(n) e^{-B_t^n} \big| \psi^n - \psi \big|^{a-2} \dd B_t^n \\
&  + d^2 a (a-1) \, e^{-B_t^n} f^2(n) \big| \psi^n - \psi \big|^{a-2} \dd t \, .
\end{align*}
Integrating in time, taking the expected value, and finally the supremum over $t\in[0,T]$, we get
\begin{align}
\sup_{t\in[0,T]} \E\Big[ e^{-B_t^n}  \big| \psi_t^n-\psi_t \big|^a \Big] \le & \  \big| \psi^n_0-\psi_0 \big|^a + a  \lambda \E \Big[ \int_0^T   e^{-B_s^n}  \big| \psi_s^n-\psi_s \big|^a \dd s \Big] \nonumber \\
&\hspace{-2,5cm} + C_{a,d,\lambda} f(n) \E \Big[  \int_0^T  e^{-B_s^n} \Big(  \big| \psi_s^n-\psi_s \big|^{a-1} + f(n) \big| \psi_s^n-\psi_s \big|^{a-2} \Big) \dd s  \Big]  \nonumber \\
&\hspace{-2,5cm} + f^2(n) \E \Big[ \int_0^T e^{-B_s^{n}}  \big| \psi^n_s - \psi_s \big|^{a-2} \dd B_s^n \Big] \, .  \label{termine claim}
\end{align}
The expected value in the second line above is bounded uniformly in $n$. This fact is easily seen using for each term H\"older inequality together with the integrability properties of the flows $\psi^n$ and of the exponential of the processes $B_s^n$, provided by (\ref{integrability Y}) and Lemma \ref{lem e^A^n} respectively. We claim that also the expected value of the last line is bounded.

\begin{claim}
There exists a constant $C$ s.t. for every $n$ and $p\ge0$
\begin{equation*}
\E \Big[ \int_0^T e^{-B_s^n}  \big| \psi^n_s - \psi_s \big|^{p} \dd B_s^n \Big] \le C \ .
\end{equation*}
\end{claim}

\begin{proof}[Proof of the Claim]
Using the definition of $B_t^n$ we can rewrite the term on the left hand side as
$$ \E \Big[ \int_0^T  e^{-B_s^n}   \big| \psi^n_s - \psi_s \big|^{p}  \frac{\big| \nabla U(s,\phi^n_s) - \nabla U(s,\phi_s ) \big|^2}{ \big| \phi^n_s - \phi_s \big|^2}   \dd s  \Big]  \ .   $$
Using H\"older inequality, for some $\varepsilon>1$ small (to be fixed later) and $k$ the conjugate exponent, we obtain the term
\begin{equation*}
\E \Big[ \int_0^T  e^{-k B_s^n}  \big| \psi^n_s - \psi_s \big|^{k p}   \dd s  \Big] \ , 
\end{equation*}
for which we have already obtained a uniform bound, and the term
\begin{equation}\label{termine brutto stima B^n}
\E \Big[ \int_0^T    \frac{\big| \nabla U^n(s,\phi^n_s) - \nabla U^n(s,\phi_s ) \big|^{2\varepsilon}}{ \big| \phi^n_s - \phi_s \big|^{2\varepsilon}}      \dd s  \Big] \ .
\end{equation}
For this term, we proceed as in the proof of Lemma \ref{lem e^A^n}. The key point is the estimate of the term
\begin{equation*}
\int_0^1 \E \Big[ \int_0^T   \big| \nabla^2 U^n(s,\phi^{n,r}_s) \big|^{2\varepsilon}     \dd s  \Big] \dd r \ ,
\end{equation*}
where
\begin{equation*}
\phi_t^{n,r} = rx+(1-r)y + \int_0^t r b^n (s,\phi^n_s) + (1-r) b(s,\phi_s) \dd s + W_t \ .
\end{equation*}
We can conclude as in the proof of Lemma \ref{lem e^A^n} if we use the result of Lemma \ref{lemma integrability b(X)}. In particular, (\ref{termine brutto stima B^n}) is controlled by $\|b\|_{L_p^q}$. \ \ \ 
\end{proof}\\

We return to the proof of Lemma \ref{lemma convergenza SDE}. Thanks to the uniform bounds obtained for the expectations in the second and third lines of (\ref{termine claim}), we can pass to the limit in $n$ to obtain
\begin{align*}
\limsup_{n} \sup_{t\in[0,T]} \E \Big[ e^{-B_t^n}  \big| \psi_t^n-\psi_t \big|^a \Big] &\\
& \hspace{-3cm} \le  \limsup_n C_{a}\big(|\phi_0(x)-\phi_0(y)|^{a}+f(n)^a \big) \\
& \hspace{-3cm} \ \ + C_{a,\lambda} \limsup_n  \E \Big[ \int_0^T   e^{-B_s^n}  \big| \psi_s^n-\psi_s \big|^a \dd s \Big]\\
& \hspace{-3cm}  \le  C_{a}|x-y|^{a} + C_{a,\lambda} \int_0^T \limsup_n \sup_{t\in[0,s]}  \E \Big[  e^{-B_t^n}  \big| \psi_t^n-\psi_t \big|^a \Big] \dd s \, .
\end{align*}
Using Gronwall lemma we get
\begin{equation}\label{eq limsup}
 \limsup_{n} \sup_{t\in[0,T]}\E \Big[ e^{-B_t^n}  \big| \psi_t^n-\psi_t \big|^a \Big]  \le C_{a,\lambda,T} \,  |x-y|^{a} \ .
\end{equation}
We can now get rid of the exponential factor using again H\"older inequality
\begin{align*}
\limsup_n \sup_{t\in[0,T]} \E\Big[ |\psi^n - \psi |^{a} \Big] &\le \limsup_n\bigg\{ \E \Big[ e^{2 B_T^n} \Big]^{1/2} \hspace{-0,3cm} \sup_{t\in[0,T]}  \hspace{-0,1cm} \E \Big[ e^{-2 B_t^n}  \big| \psi_t^n-\psi_t \big|^{2 a} \Big]^{1/2}  \bigg\} \\
&\le C_{p,\lambda, T} \, |x-y|^a \ .
\end{align*}
With $a=2p$, redefining $B_t^n$ as $1/2$ of the process defined above and using the relation (\ref{X-Y}), we  can finally transport this bound to the flows $\phi^n$:
\begin{align*}
\limsup_n \sup_{t\in[0,T]} \E\Big[ |\phi^n_t - \phi_t |^{p} \Big] &\le C_{p} \limsup_n \Big( \sup_{t\in[0,T]}\E\Big[ |\psi^n_t - \psi_t |^{p} \Big] + f^p(n) \Big) \\
&\le C_{p,\lambda, T}  |x-y|^p \ .
\end{align*}
Remark that all the estimates found are uniform in $x,y\in B(0,R)$, so that we have obtained the desired result for the forward flows. But since the backward flows $\phi^{t,n}_0(\cdot)$ and $\phi^t_0(\cdot)$ are solution of the same SDE driven by the drifts $-b^n$ and $-b$, the same result holds for them too. \ \ \ 
\end{proof}

\begin{lemma} \label{lemma stima derivata}
For every $p\geq1$, there exists $C_{d,p,T}>0$ such that
\begin{equation}
\sup_{t\in\left[  0,T\right]  }\sup_{x\in \R^d }E\left[ \left\vert \nabla\phi_{0}^{t,n}\left(  x\right)  \right\vert ^{p}\right]  \leq C_{d,p,T}  \label{property 2}
\end{equation}
uniformly in $n$.
\end{lemma}

\begin{proof}
Again, since the backward flow satisfies the same SDE of the forward flow with a drift of opposite sign, it is enough to show that the uniform bound (\ref{property 2}) holds for the forward flows.
Let $\theta^n$ and $\xi^n$ be the derivatives of $\phi^n$ and $\psi^n$, respectively. Since $\phi^n_t = (\gamma_t^n)^{-1} \circ \psi _t \circ \gamma_0^n$, from (\ref{stima invertibilita}) we have $|\theta^n_i |^p \le C_{d,p} |\xi^n|^p $. Therefore, we only need to show that the estimate (\ref{property 2}) holds for the flow $\psi^n$, which solves
\begin{equation*}
\dd \xi_t^{n} (x)= \lambda \nabla U^n \big(  t,\phi_t^n(x) \big) \, \xi_t^{n} (x)\, \dd t +  \nabla^2 U^n \big(  t,\phi_t^n (x)  \big) \,  \xi_t^{n} (x)  \,\dd W_t  
\end{equation*}
with initial condition $\xi_0^{n} (x) = Id$. For the rest of the proof we take any fixed $x\in \R^d$. $\nabla U^n$ is bounded uniformly in $n$ and the function $\nabla^2 U$ is at least in $L_p^q$, so that the last term is the differential of a martingale $(\dd M_t^{n})$ due to Lemma \ref{lemma integrability b(X)}. By It\^o formula we have therefore
\begin{equation*}
\dd |\xi^n|^p \le C |\xi^n_t|^p \dd t + \dd M_t^n + |\xi^n_t|^{p-2} \mathrm{Tr}\Big( \big[ \nabla^2 U^n \left(  t,\phi_t^n (x)  \right)  \xi_t^{n} \big]\big[ \nabla^2 U^n \left(  t,\phi_t^n (x)  \right)  \xi_t^{n} \big]^T\Big) \dd t \ .
\end{equation*}
The constant $C$ can be chosen independently of $n$, and the trace of the matrix in the last term above can be controlled by a constant $C_{p,d}$\,, depending on $p$ and the dimension $d$ of the space, times $|\xi^n_t|^2 |\nabla^2U^n(t, \phi_t^n)|^2$. Introduce the process
\begin{equation}\label{new A^n}
A_t^n := C_{p,d} \int_0^t \big| \nabla^2U^n\big(s, \phi_s^n \big) \big|^2 \dd s \ .
\end{equation}
This is a continuous, adapted, non decreasing process, with $A^n_0=0$ and, due to Lemma \ref{lemma integrability b(X)}, $\E[A^n_T]\le C$ uniformly in $n$. Lemma \ref{lem e^A^n} even provides the bound $\E\big[ e^{k A_t^n} \big] \le C_{\|U^n\|}$ for any real constant $k$. We can therefore find a bound uniform in $n$ reasoning as in Lemma \ref{lemma U^n}. We find that
\begin{equation*}
\dd \, e^{-A^n_t} |\xi_t^n|^p \le C e^{-A^n_t} |\xi_t^n|^p \, \dd t + e^{-A^n_t}  \, \dd M_t^n 
\end{equation*}
and after integrating and taking the expected value one obtains
\begin{equation*}
\E \big[e^{-A^n_t} |\xi_t^n|^p   \big] \le |\xi_0|^p + C \int_0^t \E \big[ e^{-A^n_s} |\xi_s^n|^p \big] \, \dd s \, .
\end{equation*}
Take the supremum over all $t\in[0,T]$ and apply Gronwall inequality to get
\begin{equation*}
\sup_{t\in[0,T]} \E \big[e^{-A^n_t} |\xi_t^n|^p  \big] \le C_T  |\xi_0^n|^p = C_{d,p,T} \ ,
\end{equation*}
uniformly in $n$ and $x\in\R^d$. Using H\"older inequality as in the proof of the previous lemma, we finally obtain estimate (\ref{property 2}) for the derivative of the flow $\psi^n$, and this concludes the proof. \ \ \ 
\end{proof}

\section{Main Result of Existence of Weakly Differentiable Solutions}\label{sec existence}

Consider the SPDE in Stratonovich form%
\[
\frac{\partial u}{\partial t}+b\cdot\nabla u+\sigma\nabla u\circ\frac
{\dd W}{\dd t}=0 \, ,   \qquad   u|_{t=0}=u_{0}  \, .
\]
The It\^{o} formulation (as explained in detail also in \cite{FGP1}) is%
\[
\dd u+b\cdot\nabla u \, \dd t+\sigma\nabla u \, \dd W = \frac{\sigma^{2}}{2} \Delta u \, \dd t \, ,    \qquad   u|_{t=0}=u_{0} \, .
\]
In this section we assume $b\in L_p^q$, with $p,q$ satisfying condition (\ref{pq}).

\begin{definition}
\label{main definition}Assume that $b\in L_p^q$, with $p,q$ as in (\ref{pq}). We say
that $u$ is a weakly differentiable solution of the SPDE if

\begin{enumerate}
\item $u:\Omega\times\left[  0,T\right]  \times\mathbb{R}^{d}\rightarrow
\mathbb{R}$ is measurable, $\int u\left(  t,x\right)  \varphi\left(  x\right)
\dd x$ (well defined by property 2 below) is progressively measurable for each
$\varphi\in C_{0}^{\infty}\left(  \mathbb{R}^{d}\right) ; $

\item $P\left(  u\left(  t,\cdot\right)  \in\cap_{r\geq1}W_{loc}^{1,r}\left(
\mathbb{R}^{d}\right)  \right)  =1$ for every $t\in\left[  0,T\right]  $ and  
both $u$ and $\nabla u$ are in $\calC^0\big( [0,T] ; \cap_{r\geq1}L^{r} (\Omega \times \R^d) \big) ;$

\item for every $\varphi\in C_{0}^{\infty}\left(  \mathbb{R}^{d}\right)  $ and
$t\in\left[  0,T\right]  $, with probability one one has%
\begin{align*}
&  \int u(t,x)  \varphi(x) \, \dd x + \int_{0}^{t}\int  b(s,x)  \cdot \nabla u(s,x) \varphi(x) \,  \dd x \dd s\\
&  =\int u_{0} (x)  \varphi(x) \, \dd x + \sigma \sum_{i=1}^{d} \int_{0}^{t}\left(  \int u \left(  s,x\right)  \partial_{x_{i}}  \varphi\left(
x\right) \, \dd x \right)  \dd W_{s}^{i}\\
&  +\frac{\sigma^{2}}{2}\int_{0}^{t}\int u(s,x) \Delta \varphi (x) \, \dd x \dd s \, .
\end{align*}

\end{enumerate}
\end{definition}

\begin{remark}
The process $s\mapsto Y_{s}^{i}:=\int u\left(  s,x\right)  \partial_{x_{i}%
}\varphi\left(  x\right)  \dd x$ is progressively measurable by property 1 and
satisfies $\int_{0}^{T}\left\vert Y_{s}^{i}\right\vert ^{2} \dd s < \infty$ by
property 2, hence the It\^{o} integral is well defined.
\end{remark}

\begin{remark}
The term $\int_{0}^{t}\int b\left(  s,x\right)  \cdot\nabla u\left(
s,x\right)  \varphi\left(  x\right)  \dd x \dd s$ is well defined with probability
one because of the integrability properties in $\left(  t,x\right)  $ of
$b~$(assumptions) and $\nabla u$ (property 2).
\end{remark}

\begin{remark}
From 3 it follows that $\int u\left(  t,x\right)  \varphi\left(  x\right)  \dd x$
has a continuous adapted modification, for every $\varphi\in C_{0}^{\infty
}(  \mathbb{R}^{d})  $.
\end{remark}

Let $\phi_{t}\left(  \omega\right)  :\mathbb{R}^{d}\rightarrow\mathbb{R}^{d}$
be the $\alpha$-H\"{o}lder continuous stochastic flow of homeomorphisms, for
every $\alpha\in\left(  0,1\right)  $, associated to the SDE%
\[
\dd X_{t}^{x} = b\left(  t,X_{t}^{x}\right)  \dd t + \dd W_{t} \, ,      \qquad     X_{0}^{x} = x
\]
constructed in \cite{FeFl2}. The inverse of $\phi_{t}$ will be denoted by $\phi_{0}^{t}$.

\begin{theorem}\label{main teo}
Assume $b\in L_p^q$ with $p,q$ as in (\ref{pq}). If $u_{0}\in \cap_{r\geq1} W^{1,r} ( \mathbb{R}^{d})$
then $u\left(  t,x\right)  :=u_{0}\left(  \phi_{0}^{t}(x) \right)  $ is a weakly differentiable solution of the SPDE.
\end{theorem}


\begin{proof}
\textbf{Step 1} (preparation). The random field $(\omega,t,x)
\mapsto u_{0}\left(  \phi_0^t(x)  \right)  $ is jointly
measurable and $( \omega,t)  \mapsto\int u_0 \left( \phi_0
^t(x) \right)  \varphi\left(x\right) \dd x$ is progressively
measurable for each $\varphi\in C_{0}^{\infty}(  \mathbb{R}^{d})
$. Hence part 1 of Definition \ref{main definition} is true. We could prove
part 2 by chain rule and Sobolev properties of $\phi_{0}^{t}\left(  x\right)
$. However, a direct verification of part 3 from the formula $u\left(
t,x\right)  :=u_{0}\left(  \phi_{0}^{t}\left(  x\right)  \right)  $ is
difficult because of lack of calculus. Hence we choose to approximate
$u\left(  t,x\right)  $ by a smooth field $u_{n}\left(  t,x\right)  $; doing
this, we prove both 2 and 3 by means of this approximation.

Let $u_{0}^{n}$ be a sequence of smooth functions which converges to $u_{0}$ in $W^{1,r}(\R^d)$ and
uniformly on $\mathbb{R}^{d}$. It is easy to check that these properties
are satisfied for instance by $u_{0}^{n}\left(  x\right)  =\int\theta
_{n}\left(  x-y\right)  u_{0}\left(  y\right)  \dd y$ when $\theta_{n}$ are usual
mollifiers; for instance, the uniform convergence property comes from
\begin{align*}
\left\vert u_{0}^{n}\left(  x\right)  -u_{0}\left(  x\right)  \right\vert  &
\leq  \int\theta_{n}\left(  x-y\right)  \left\vert u_{0}\left(  y\right)
-u_{0}\left(  x\right)  \right\vert \dd y \\
& \leq C\int\theta_{n}\left(  x-y\right)  \left\vert y-x\right\vert^\alpha 
\dd y = C\int\theta_{n}\left(  y\right)  \left\vert y\right\vert^\alpha \dd y
\end{align*}
because $u_{0}\in \calC^{0,\alpha}$. 

Let $\phi_{t}^{n}\left(  \omega\right)  :\mathbb{R}^{d}\rightarrow
\mathbb{R}^{d}$ be the stochastic flow of smooth diffeomorphisms associated to
the equation
\[
\dd X_{t}^{x,n} = b_{n}\left(  t,X_{t}^{x,n}\right)  \dd t + \dd W_{t} \, ,    \qquad    X_{0}^{x,n} = x  \, ,
\]
where $b_{n}$ are smooth approximations of $b$ as considered in the previous
section, and let $\phi_{0}^{t,n}$ be the inverse of $\phi_{t}^{n}$. Then
$u_{n}\left(  t,x\right)  :=u_{0}^{n}\left(  \phi_{0}^{t,n}\left(  x\right)
\right)  $ is a smooth solution of%
\[
\dd u_{n} + b_{n} \cdot \nabla u_{n}  \, \dd t  + \sigma\nabla u_{n} \, \dd W  =  \frac{\sigma
^{2}}{2}\Delta u_{n} \,  \dd t \, ,     \qquad      u_{n}\big|_{t=0}=u_{0}^{n}  \, , %
\]
see \cite[Theorem 6.1.5]{Ku}, and thus it satisfies                   
\begin{align*}
&  \int u_{n}\left(  t,x\right)  \varphi\left(  x\right)  \dd x +\int_{0}^{t}\int
b_{n}\left(  s,x\right)  \cdot\nabla u_{n}\left(  s,x\right)  \varphi\left(
x\right)  \dd x \dd s\\
&  =\int u_{0}^{n}\left(  x\right)  \varphi\left(  x\right)  \dd x  + \sigma \sum_{i=1}%
^{d}\int_{0}^{t}\left(  \int u_{n}\left(  s,x\right)  \partial_{x_{i}}%
\varphi\left(  x\right)  \dd x \right)  \dd W_{s}^{i}\\
&  +\frac{\sigma^{2}}{2}\int_{0}^{t}\int u_{n}\left(  s,x\right)
\Delta\varphi\left(  x\right)  \dd x \dd s
\end{align*}
for every $\varphi\in C_{0}^{\infty}(  \mathbb{R}^{d})  $ and
$t\in\left[  0,T\right]  $, with probability one. We need to establish
suitable bounds on $u_{n}\left(  t,x\right)  $ and suitable convergence
properties of $u_{n}\left(  t,x\right)  $ to $u\left(  t,x\right)  $ in order
to apply Lemma \ref{lemma appendix} - which is the first step to obtain the regularity properties of $u$ of point 2 of Definition \ref{main definition} - and pass to the limit in the equation. More precisely, for every $\varphi\in C_{0}
^{\infty}(  \mathbb{R}^{d})  $, $t\in\left[  0,T\right]  $ and
bounded r.v. $Z$ we have
\begin{align*}
&  \E \left[  Z \int u_{n}\left(  t,x\right)  \varphi\left(  x\right)  \dd x\right]
+ \E \left[  Z\int_{0}^{t}\int b_{n}\left(  s,x\right)  \cdot\nabla u_{n}\left(
s,x\right)  \varphi\left(  x\right)  \dd x \dd s \right]  \\
&  =\int u_{0}^{n}\left(  x\right)  \varphi\left(  x\right)  \dd x + \sigma \sum_{i=1}%
^{d}  \E \left[  Z\int_{0}^{t}\left(  \int u_{n}\left(  s,x\right)  \partial
_{x_{i}}\varphi\left(  x\right)  \dd x \right)  \dd W_{s}^{i} \right]
\end{align*}%
\begin{equation}
+\frac{\sigma^{2}}{2} \E \left[  Z\int_{0}^{t}\int u_{n}\left(  s,x\right)
\Delta\varphi\left(  x\right)  \dd x \dd s \right]  .\label{very weak formulation}%
\end{equation}
We shall pass to the limit in each one of these terms. We are forced to use
this very weak convergence due to the term
\[
\E \left[  Z\int_{0}^{t}\int b_{n}\left(  s,x\right)  \cdot\nabla u_{n}\left(
s,x\right)  \varphi\left(  x\right)  \dd x \dd s \right]
\]
where we may only use weak convergence of $\nabla u_{n}$.

\textbf{Step 2} (convergence of $u_{n}$ to $u$). We claim that, uniformly in $n$ and for every $r\ge1$,
\begin{align}
\sup_{t\in[0T]} \int_{\R^d} \E \Big[ |u_n(t,x)|^r \Big] \dd x \le C_r \ , \label{stim unif un}\\
\sup_{t\in[0T]} \int_{\R^d} \E \Big[ |\nabla u_n(t,x)|^r \Big] \dd x \le C_r \ . \label{sima unif nabla un}
\end{align}
Let us show how to prove the second bound; the first one can be obtained in the same way. We use the representation formula for $u_n$ and H\"older inequality to obtain
\begin{equation*}
\bigg(\int_{\R^d}  \hspace{-2mm}  \E \Big[ |\nabla u_n(t,x)|^r \Big] \dd x \bigg)^2 \hspace{-0.1cm} \le \hspace{-0.5mm} \int_{\R^d}   \hspace{-2mm}  \E \Big[ \big| \nabla u_0^n\big( \phi_0^{t,n}(x) \big) \big|^{2r} \Big] \dd x \int_{\R^d}  \hspace{-2mm} \E \Big[\big| \nabla \phi_0^{t,n}(x) \big|^{2r} \Big] \dd x \, .
\end{equation*}
The last integral on the right hand side is uniformly bounded by (\ref{property 2}). Also the other integral term can be bounded uniformly: changing variables (recall that all functions involved are regular) we get
\begin{equation*}
\int_{\R^d}  \E \Big[ \big| \nabla u_0^n\big( \phi_0^{t,n}(x) \big) \big|^{2r} \Big] \dd x \le \int_{\R^d}   \big| \nabla u_0^n (y) \big|^{2r} \E\big[  J_{\phi_t^n(y)}  \big] \dd y \ ,
\end{equation*}
where $J_{\phi_t^n(y)} $ is the Jacobian of $\phi_t^n(y)$; this last term can be controlled using again H\"older inequality, (\ref{property 2}) and the convergence of $u_0^n$ in $W^{1,r}$ (for every $r\ge1$). Remark that all the bounds obtained are uniform in $n$ and $t$.

We consider now the problem of the convergence of $u_n$ to $u$. Let us first prove that, given $t\in\left[  0,T\right]  $ and $\varphi\in C_{0}^{\infty}(\mathbb{R}^{d})  $,
\begin{equation}
P-\lim_{n\rightarrow\infty}\int_{\mathbb{R}^{d}}u_{n}\left(  t,x\right)
\varphi\left(  x\right)  \dd x  =  \int_{\mathbb{R}^{d}}u\left(  t,x\right)
\varphi\left(  x\right)  \dd x \, \label{convergence 1}
\end{equation}
(convergence in probability). This is the first assumption of Lemma
\ref{lemma appendix} and allows also to pass to the limit in the first term of
equation (\ref{very weak formulation}) using the uniform bound (\ref{stim unif un}) and Vitali convergence theorem (we are on the compact support of the test function $\varphi$). Since
\begin{align*}
\int_{\mathbb{R}^{d}}u_{n}\left(  t,x\right)  \varphi\left(  x\right)  \dd x  &
=\int_{\mathbb{R}^{d}}u_{0}^{n}\left(  \phi_{0}^{t,n}\left(  x\right)
\right)  \varphi\left(  x\right)  \dd x \\
& \hspace{-1cm} = \hspace{-0,05cm} \int_{\mathbb{R}^{d}}  \hspace{-0,05cm} \left(  u_{0}^{n}-u_{0}\right)  \left(  \phi
_{0}^{t,n} \left(  x\right)  \right)  \varphi\left(  x\right)  \dd x  +  \hspace{-0,05cm} \int%
_{\mathbb{R}^{d}}  \hspace{-0,1cm} u_{0}\left(  \phi_{0}^{t,n}\left(  x\right)  \right)
\varphi\left(  x\right)  \dd x \ ,
\end{align*}
using Sobolev embedding $W^{1,2d}\hookrightarrow \calC^{0,1/2}$  we have%
\begin{align*}
\left\vert \int_{\mathbb{R}^{d}}\left(  u_{n}\left(  t,x\right)  -u\left(
t,x\right)  \right)  \varphi\left(  x\right)  \dd x \right\vert  &  \leq\left\Vert
u_{0}^{n}-u_{0}\right\Vert _{L^{\infty}}\left\Vert \varphi\right\Vert _{L^{1}%
}\\
&  + C \left\Vert \varphi\right\Vert _{L^{\infty}}\int_{B_R
}\left\vert \phi_{0}^{t,n}\left(  x\right)  -\phi_{0}^{t}\left(  x\right)
\right\vert^{1/2} \dd x \, .     
\end{align*} 
The first term converges to zero by the uniform convergence of $u_{0}^{n}$ to 
$u_{0}$. To treat the second one, recall we have proved property
(\ref{property 1}). Hence%
\[
\lim_{n\rightarrow\infty}  \E \left[  \int_{B_R  }\left\vert
\phi_{0}^{t,n}\left(  x\right)  -\phi_{0}^{t}\left(  x\right)  \right\vert
\dd x \right]  =0
\]
and thus
\[
P-\lim_{n\rightarrow\infty}\int_{B_R  }\left\vert \phi
_{0}^{t,n}\left(  x\right)  -\phi_{0}^{t}\left(  x\right)  \right\vert \dd x =0.
\]
Property (\ref{convergence 1}) is proved. 

Similarly, we can show that, given $\varphi\in C_{0}^{\infty}\left(
\mathbb{R}^{d}\right)  $,
\begin{equation}
P-\lim_{n\rightarrow\infty}\int_{0}^{T}\left\vert \int_{\mathbb{R}^{d}}\left(
u_{n}\left(  t,x\right)  -u\left(  t,x\right)  \right)  \varphi\left(
x\right)  \dd x \right\vert ^{2}  \dd t  = 0 \, .\label{property 3b}%
\end{equation}
This implies that we can pass to the limit in the last two terms of equation
(\ref{very weak formulation}). Indeed, property (\ref{property 3b}) implies
that%
\[
P-\lim_{n\rightarrow\infty}\int_{0}^{t} \hspace{-0,05cm}  \left(  \int \hspace{-0,05cm}  u_{n}\left(  s,x\right)
\partial_{x_{i}}\varphi\left(  x\right)  \dd x \right) \hspace{-0,05cm}  \dd W_{s}^{i}   = \hspace{-0,05cm}  \int_{0}^{t} \hspace{-0,05cm}  \left(  \int  \hspace{-0,05cm}  u\left(  s,x\right)  \partial_{x_{i}}\varphi\left(  x\right)
\dd x \right) \hspace{-0,05cm}  \dd W_{s}^{i}%
\]
for each $i=1,...,d$. Moreover,
\[
\E \left[  \left\vert \hspace{-0,05cm} \int_{0}^{t} \hspace{-0,1cm} \left(  \int \hspace{-0,05cm} u_{n}\left(  s,x\right)
\partial_{x_{i}}\varphi\left(  x\right)  \dd x \right) \hspace{-0,05cm}  \dd W_{s}^{i}\right\vert
^{2}\right]   \hspace{-0,05cm}  = \E \left[  \int_{0}^{t}\left\vert \int \hspace{-0,05cm} u_{n}\left(  s,x\right)
\partial_{x_{i}}\varphi\left(  x\right)  \dd x \right\vert ^{2} \hspace{-0,1cm}  \dd s \right]  ,
\]
which is uniformly bounded in $n$ due to (\ref{stim unif un}). By Vitali convergence theorem we obtain that
\begin{align*}
&\lim_{n\rightarrow\infty}   \E   \left[  Z  \int_{0}^{t}  \left(  \int   u_{n}\left(
s,x\right)  \partial_{x_{i}}\varphi\left(  x\right)  \dd x \right)   \dd W_{s}^{i}\right]  \\
& \hspace{6cm} =   \E   \left[  Z  \int_{0}^{t}   \left(  \int   u\left(  s,x\right)
\partial_{x_{i}}\varphi\left(  x\right)  \dd x \right)   \dd W_{s}^{i}\right] .
\end{align*}
The proof of convergence for the last term of equation
(\ref{very weak formulation}) is similar.

\textbf{Step 3} (regularity of $u$). Let us prove property 2 of Definition
\ref{main definition}. The key estimate is property (\ref{property 2}). 

Given $r\geq1$ and $t\in\left[  0,T\right]  $, let us prove that $P\left(
u\left(  t,\cdot\right)  \in W_{loc}^{1,r}(  \mathbb{R}^{d})
\right)  =1$. We want to use Lemma \ref{lemma appendix} with $F=u$,
$F_{n}=u_{n}$. Condition 1 of Lemma \ref{lemma appendix} is provided by (\ref{convergence 1}). It is clear that $u_{n}\left(
t,\cdot\right)  \in W_{loc}^{1,r}(  \mathbb{R}^{d})  $ for $P$-a.e.
$\omega$, so that condition 2 follows from the uniform bound on $\nabla u_n$ obtained in (\ref{sima unif nabla un}).
We can apply Lemma \ref{lemma appendix} and get $u\left(  t,\cdot\right)
\in W_{loc}^{1,r}(  \mathbb{R}^{d})  $ for $P$-a.e. $\omega$.

Let us prove the second part of property 2 of Definition \ref{main definition}.
We have, from Lemma \ref{lemma appendix} and (\ref{sima unif nabla un}), 
\[
\E \left[  \int_{B_R  }\left\vert \nabla u\left(  t,x\right)
\right\vert ^{r}  \dd x \right]  \leq  \underset{n\rightarrow\infty}{\lim\sup} \ \E \left[
\int_{B_R  }\left\vert \nabla u_{n}\left(  t,x\right)
\right\vert ^{r}  \dd x \right]  \leq C_r
\]
for every $R>0$ and $t\in[0,T]$. Hence, by monotone convergence we have
\begin{equation} \label{sima unif nabla u}
\sup_{t\in[0,T]} \E \left[  \int_{\R^d  }\left\vert \nabla u\left(  t,x\right)
\right\vert ^{r}  \dd x \right]  \leq C_r \ .
\end{equation}
A similar bound can be proved for $u$ itself: using (\ref{stim unif un}), the convergence in probability proved in the previous step and Vitali convergence theorem we get that for any $r'<r$, $R>0$ and uniformly in time,
\begin{equation*}
\int_{B_R} \E\Big[ |u(t,x) |^{r'} \Big] \dd x = \lim_{n\to\infty} \int_{B_R} \E\Big[ |u_n(t,x) |^{r'} \Big] \dd x \le C_r \ ;
\end{equation*}
by monotone convergence it follows that
\begin{equation*}
\sup_{t\in[0,T]} \int_{\R^d} \E\Big[ |u(t,x) |^{r'} \Big] \dd x \le C_r\ .
\end{equation*}

\textbf{Step 4} (passage to the limit). Finally, we have to prove that we can 
pass to the limit in equation (\ref{very weak formulation}) and deduce that
$u$ satisfies property 3 of Definition \ref{main definition}. We have already
proved that all terms converge to the corresponding ones except for the term
$\E\left[  Z\int_{0}^{t}\int b_{n}\left(  s,x\right)  \cdot\nabla u_{n}\left(
s,x\right)  \varphi\left(  x\right)  \dd x \dd s \right]  $. We do not want to
integrate by parts, for otherwise we would have to assume something on
$\operatorname{div}b$. Since $b_{n}\rightarrow b$ in $L_{p}^{q}=L^{q}\left(
\left[  0,T\right]  ;L^{p}(  \mathbb{R}^{d})  \right)  $, it is
sufficient to use a suitable weak convergence of $\nabla u_{n}$ to $\nabla u$. Precisely,
\begin{align*}
& \E\left[  Z\int_{0}^{t}\int b_{n}\left(  s,x\right)  \cdot\nabla u_{n}\left(
s,x\right)  \varphi\left(  x\right)  \dd x \dd s \right] \\
& \hspace{2cm} -  \E \left[  Z\int_{0}^{t}\int
b\left(  s,x\right)  \cdot\nabla u\left(  s,x\right)  \varphi\left(  x\right)
\dd x \dd s \right]  =   I_{n}^{\left(  1\right)  }\left(  t\right)  +   I_{n}^{\left(  2\right)
}\left(  t\right)
\end{align*}%
\begin{align*}
I_{n}^{\left(  1\right)  }\left(  t\right)    & = \E \left[  Z \int_{0}^{t}%
\int  \big(  b_{n}\left(  s,x\right)  -b\left(  s,x\right)  \big)
\cdot\nabla u_{n}\left(  s,x\right)  \varphi\left(  x\right)  \dd x \dd s \right]  \\
I_{n}^{\left(  2\right)  }\left(  t\right)    & = \E \left[  Z\int_{0}^{t}%
\int\varphi\left(  x\right)  b\left(  s,x\right)  \cdot   \big(  \nabla
u_{n}\left(  s,x\right)  -\nabla u\left(  s,x\right)  \big)  \dd x \dd s \right]  .
\end{align*}
We have to prove that both $I_{n}^{\left(  1\right)  }\left(  t\right)  $ and
$I_{n}^{\left(  2\right)  }\left(  t\right)  $ converge to zero as
$n\rightarrow\infty$.  By H\"{o}lder inequality,
\[
I_{n}^{\left(  1\right)  }\left(  t\right)  \leq C\left\Vert b_{n}%
-b\right\Vert _{L^{q}\left(  \left[  0,T\right]  ;L^{p}(  \mathbb{R}%
^{d})  \right)  }   \E \left[  \left\Vert \nabla u_{n}\right\Vert
_{L^{q^{\prime}}\left(  \left[  0,T\right]  ;L^{p^{\prime}}(  \R^d )  \right)  }\right]
\]
where $1/p+1/p^{\prime}=1$ and $1/q+1/q^{\prime}=1$. 
Thus, from (\ref{sima unif nabla un}), $I_{n}^{\left(  1\right)  }\left(
t\right)  $ converges to zero.

Let us treat $I_{n}^{\left(  2\right)  }\left(  t\right)  $. Using the
integrability properties shown above we have%
\begin{align*}
& \E  \left[  Z\int_{0}^{t}\int\varphi\left(  x\right)  b\left(  s,x\right)
\cdot   \big(  \nabla u_{n}\left(  s,x\right)  -\nabla u\left(  s,x\right)
\big) \, \dd x \dd s \right]  \\
& \hspace{2cm} =\int_{0}^{t}  \E \left[  \int Z\varphi\left(  x\right)  b\left(  s,x\right)
\cdot   \big(  \nabla u_{n}\left(  s,x\right)  -\nabla u\left(  s,x\right)
\big) \, \dd x \right]  \dd s \, .
\end{align*}
The function%
\[
h_{n}\left(  s\right)  :=  \E \left[  \int Z\varphi\left(  x\right)  b\left(
s,x\right)  \cdot   \big(  \nabla u_{n}\left(  s,x\right)  -\nabla u\left(
s,x\right)  \big) \, \dd x \right]
\]
converges to zero as $n\rightarrow\infty$ for almost every $s$ and satisfies
the assumptions of Vitali convergence theorem (we shall prove these two claims
in Step 5 below). Hence $I_{n}^{\left(  2\right)  }\left(  t\right)  $
converges to zero. 

Now we may pass to the limit in equation (\ref{very weak formulation}) and get%
\begin{align*}
&  \E \left[  Z\int u\left(  t,x\right)  \varphi\left(  x\right)  \dd x \right]
+  \E \left[  Z\int_{0}^{t}\int b\left(  s,x\right)  \cdot\nabla u\left(
s,x\right)  \varphi\left(  x\right)  \dd x \dd s \right]  \\
&  =   \int u_{0}\left(  x\right)  \varphi\left(  x\right)  \dd x  +  \sigma \sum_{i=1}%
^{d} \E \left[  Z\int_{0}^{t}\left(  \int u\left(  s,x\right)  \partial_{x_{i}%
}\varphi\left(  x\right)  \dd x \right)  \dd W_{s}^{i}\right]  \\
&  +  \frac{\sigma^{2}}{2} \E \left[  Z\int_{0}^{t}\int u\left(  s,x\right)
\Delta\varphi\left(  x\right)  \dd x \dd s \right]
\end{align*}
The arbitrariness of $Z$ implies property 3 of Definition
\ref{main definition}.

\textbf{Step 5} (auxiliary facts). We have to prove the two properties of
$h_{n}\left(  s\right)  $ claimed in Step 4. Recall we may use Lemma
\ref{lemma appendix} at each value of time. It gives us
\begin{equation}
\E \left[  \int_{\mathbb{R}^{d}}\partial_{x_{i}}u\left(  s,x\right)
\varphi\left(  x\right)  Z \, \dd x \right]  =  \lim_{n\rightarrow\infty}  \E \left[
\int_{\mathbb{R}^{d}}\partial_{x_{i}}u_{n}\left(  s,x\right)  \varphi\left(
x\right)  Z  \, \dd x \right]  \label{convergence 4}%
\end{equation}
for every $\varphi\in C_{0}^{\infty}( \mathbb{R}^{d})$ and
bounded r.v. $Z$, at each $s\in[  0,T]  $. We have $b\in
L^{q}\left(  \left[  0,T\right]  ;L^{p}(  \mathbb{R}^{d})  \right)
$, hence $b\left(  s,\cdot\right)  \in L^{p}(  \mathbb{R}^{d})$
for a.e. $s\in\left[  0,T\right]  $. The space $C_{0}^{\infty}(
\mathbb{R}^{d})$ is dense in $L^{p}(  \mathbb{R}^{d})  $.
We may extend the convergence property (\ref{convergence 4}) to all
$\varphi\in L^{p}(  \mathbb{R}^{d})  $ by means of the bounds
(\ref{sima unif nabla un}) and (\ref{sima unif nabla u}).
Hence
$h_{n}\left(  s\right)  \rightarrow0$ as $n\rightarrow\infty$, for a.e.
$s\in\left[  0,T\right]  $. 

Moreover, for every $\varepsilon>0$ there is a constant $C_{Z, \varphi, \varepsilon}$ such that
\begin{align*}
\int_{0}^{T} \hspace{-0,3cm} h_{n}^{1+\varepsilon}(s) \dd s  & \leq C_{Z,\varphi
,\varepsilon} \hspace{-1,5mm} \int_{0}^{T} \hspace{-0,3cm} \E \hspace{-0,5mm} \bigg[  \int_{B_R  } \hspace{-0,3cm} \big|
b(s,x) \big|^{1+\varepsilon}\hspace{-1,5mm}\cdot \hspace{-0,05cm} \Big( \big|
\nabla u_{n}(s,x) \big|^{1+\varepsilon} \hspace{-1,5mm} + \hspace{-0,5mm} \big| \nabla
u(s,x) \big|^{1+\varepsilon} \Big)  \dd x \bigg] \hspace{-0,5mm} \dd s\\
& \leq C_{Z,\varphi,\varepsilon}\left\Vert b\right\Vert _{L_{p}^{q}%
}^{1+\varepsilon}\left(  \E \int_{0}^{T}\int_{B_R }\big\vert
\nabla u_{n}\left(  s,x\right)  \big\vert ^{r} \dd x \dd s \right)  ^{\frac
{1+\varepsilon}{r}}\\
& \quad +C_{Z,\varphi,\varepsilon}\left\Vert b\right\Vert _{L_{p}^{q}}%
^{1+\varepsilon}\left(  \E \int_{0}^{T}\int_{B_R }\big\vert
\nabla u\left(  s,x\right)  \big\vert ^{r}  \dd x \dd s \right)  ^{\frac
{1+\varepsilon}{r}}%
\end{align*}
for a suitable $r$ depending on $\varepsilon$ (we have used H\"{o}lder
inequality). The bounds (\ref{sima unif nabla un}) and (\ref{sima unif nabla u})
 imply that $\int_{0}^{T}h_{n}%
^{1+\varepsilon}\left(  s\right)  \dd s$ is uniformly bounded. Hence Vitali
theorem can be applied to prove that $I_{n}^{\left(  2\right)  }\left(  t\right)
=\int_{0}^{t}h_{n}\left(  s\right)  \dd s \rightarrow 0$ as $n \rightarrow \infty$.
The proof is complete. \ \ \ 
\end{proof}

\section{Uniqueness of Weakly Differentiable Solutions} \label{sec 1!}

\begin{theorem}
Weak solutions of Definition \ref{main definition} are unique.
\end{theorem}

\begin{proof}
Let $u^{i}$ be two weakly differentiable solutions of equation
\[
\frac{\partial u}{\partial t}+b\cdot\nabla u+\sigma\nabla u\circ\frac{\dd W}%
{\dd t}=0 \, ,    \qquad   u|_{t=0}=u_{0} \, .
\]
Then $u:=u^{1}-u^{2}$ is a weakly differentiable solution of
\begin{equation}\label{eq1}
\frac{\partial u}{\partial t}+b\cdot\nabla u+\sigma\nabla u\circ\frac{\dd W}%
{\dd t}=0 \, ,    \qquad      u|_{t=0}=0 \, .
\end{equation}
We want to prove that $u$ is identically zero. We divide the proof in three steps.

\textbf{Step 1} (Equation for $u^2$)
The first step consists in proving that $u^{2}$ is also a weakly
differentiable solution of
\begin{equation}\label{eq2}
\frac{\partial u^{2}}{\partial t}+b\cdot\nabla u^{2}+\sigma\nabla u^{2}%
\circ\frac{\dd W}{\dd t}=0 \, ,    \qquad   u|_{t=0}= 0
\end{equation}
namely that
\begin{align*}
&  \int u^{2}\left(  t,x\right)  \varphi\left(  x\right)  \dd x+\int_{0}^{t}\int
b\left(  s,x\right)  \cdot\nabla u^{2}\left(  s,x\right)  \varphi\left(
x\right)  \dd x \dd s\\
&  =\sigma\sum_{i=1}^{d}\int_{0}^{t}\left(  \int u^{2}\left(  s,x\right)
\partial_{x_{i}}\varphi\left(  x\right)  \dd x\right)  \dd W_{s}^{i}\\
&  +\frac{\sigma^{2}}{2}\int_{0}^{t}\int u^{2}\left(  s,x\right)
\Delta\varphi\left(  x\right)  \dd x \dd s
\end{align*}
for any $\varphi\in\calC^\infty_0(\R^d)$. Let $\theta^{\varepsilon}$ be a sequence of standard mollifiers. From the definition of weak solution, using $\varphi_y^\varepsilon(x)=\theta^\varepsilon(y-x)$, we have
\begin{align*}
& u_\varepsilon (t,y) 
+ \int_{0}^{t} b(s,y) \cdot\nabla u_\varepsilon(s,y) \,  \dd s  \\
& \hspace{4cm} +\sigma\sum_{i=1}^{d} \int_{0}^{t}  \partial_{y_{i}} u( s,y)  \circ \dd W_{s}^{i} = \int_0^t R_\varepsilon(s,y) \, \dd s \, ,\\
& R_\varepsilon(s,y) = \Big[ \int \Big( b(s,y) - b(s,x) \Big) \nabla u(s,x) \theta^\varepsilon(x-y)  \dd x   \Big] \, .
\end{align*}
The function $u_\varepsilon$ is smooth in space. For any fixed $y$, by It\^o formula we have 
\begin{align*}
\dd u_\varepsilon^2 (t,y)  =& \ 2 u_\varepsilon(t,y) \, \dd u_\varepsilon(t,y)\\
 =& - 2 u_\varepsilon (t,y) \, b(t,y) \nabla u_\varepsilon (t,y) \, \dd t -  2 \sigma \, u_\varepsilon (t,y) \sum_{i=1}^{d} \partial_{y_{i}} u_\varepsilon( s,y)  \circ \dd W_{s}^{i} \\
&  + 2 u_\varepsilon (t,y) R_\varepsilon(t,y) \, \dd t \ 
\end{align*}
which, rewritten in the week formulation using a generic test function $\varphi$, reads
\begin{align*}
& \int u_\varepsilon^2 (t,y) \varphi(y) \, \dd y + \int_{0}^{t} \int b(s,y)  \nabla u_\varepsilon^2(s,y) \varphi(y) \, \dd y  \dd s \\
& +\sigma \sum_{i=1}^{d} \int_{0}^{t} \hspace{-1mm} \Big( \int \hspace{-1mm} \partial_{y_{i}} u_\varepsilon^2 ( s,y) \varphi(y) \,  \dd y \Big) \circ \dd W_{s}^{i} = \int_0^t \hspace{-1mm} \int \hspace{-1mm} 2 u_\varepsilon (s,y) R_\varepsilon(s,y) \varphi(y) \, \dd y \dd s \ .
\end{align*}
We want now to pass to the limit for $\varepsilon\to 0$ in the different terms. Since for every $t$, $u_\varepsilon\to u$ uniformly on compact sets, by dominated convergence the first term tends to
\begin{equation*}
\int  u^2 (t,y) \varphi(y) \, \dd y \, .
\end{equation*}
For the following terms, we consider $s$ fixed. Using H\"older inequality and the convergence of $\|\nabla u_\varepsilon \|_{L^p}\to \|\nabla u \|_{L^p}$ on compact sets (recall that $\varphi$ is of compact support) for any $p\ge1$, we have
\begin{align*}
& \int b(s,y) \varphi(y) \Big( \nabla u_\varepsilon^2(s,y) -\nabla u^2(s,y) \Big) \dd y \\
&\hspace{2cm} \le \big\| b(s,y) \varphi (y) \big\|_{L^{r'}} \big\| \nabla u_\varepsilon^2(s,y) -\nabla u^2(s,y) \big\|_{L^r} \\
&\hspace{2cm} \le C_{r,\|b\|_{L^p}} \big\| \nabla u_\varepsilon^2(s,y) -\nabla u^2(s,y) \big\|_{L^r} \longrightarrow 0\ ,
\end{align*}
which is enough to obtain the convergence of the second term. In the same way one obtains also the convergence of the third term. As for the term containing the commutator $R_\varepsilon$, we can use again H\"older inequality, the uniform convergence of $u_\varepsilon$, the equi--boundedness of $\nabla u_\varepsilon$ in $L^p$ 
 for every $p\ge1$, and the continuity in mean (for a.e. $y$) of the function $b\in L^p(\R^d)$. This proves (\ref{eq2}). \\

\textbf{Step 2} (equation for $v^2$)
We have that $u$ is a.s. continuous in space and time (and therefore locally bounded) and by definition of weak solution $\nabla u \in L^r ([0,T]\times \R^d)$ for every $r\ge1$ a.s.. It follows that
$f (s) = \int \nabla u^2 (s,y) \varphi(y) \,\dd y$ is still a.s. a function in $L^r(0,T)$. This means that, writing (\ref{eq2}) in It\^o form, the stochastic integral is a martingale and
\begin{align*}
&  \int \E\left[  u^{2}\left(  t,x\right)  \right]  \varphi(x) \, \dd x + \int_{0}^{t} \int b(  s,x) \cdot \nabla \E \left[ u^{2} \left(s,x\right)  \right]  \varphi(x)  \, \dd x \dd s \\
&  =\frac{\sigma^{2}}{2}\int_{0}^{t}\int \E\left[  u^{2}\left(  s,x\right)
\right]  \Delta\varphi(x) \, \dd x \dd s.
\end{align*}
Hence $v(t,x) = \E\left[  u^{2}\left(  t,x\right)  \right]  $ satisfies%
\begin{align*}
&  \int v(  t,x) \varphi( x) \, \dd x +  \int_{0}^{t}\int b(s,x) \cdot\nabla v( t,x)  \varphi( x) \,  \dd x \dd s\\
&  =\frac{\sigma^{2}}{2}\int_{0}^{t} \int v(t,x)  \Delta  \varphi(x)  \, \dd x \dd s
\end{align*}
and is fairly regular: $v\in\calC^0\big([0,T];W^{1,r}(\R^d)\big)$ for $r\ge1$.
This follows by H\"older inequality because
\begin{align*}
\int \big| \nabla v (t,x) \big|^r \dd x &= \int \big| \E \big[ u \nabla u \big] \big| ^r \dd x \le \int \Big( \E \big[ |u|^2 \big] \E \big[ |\nabla u|^2 \big] \Big)^{r/2} \dd x \\
& \le \Big( \int \E \big[ |u|^{2r} \big]^2 \dd x \Big)^{1/2}  \Big(\int \E \big[ |\nabla u|^{2r} \big]^2 \dd x \Big)^{1/2} \le C \ ,
\end{align*}
uniformly in $t$ (similar computations provide the same result for the function $v$). 

Thanks to its global integrability properties, using approximating functions as in the first step, one can prove that $v$ solves
\begin{align}
& \int v^{2}\left(  t,x\right)  \dd x+ \sigma^{2} \int_{0}^{t}
\int\left\vert \nabla v\left(  t,x\right)  \right\vert ^{2} \dd x \dd s  \nonumber \\
& \hspace{3cm} =-2\int_{0}^{t}\int b\left(  s,x\right)  \cdot\nabla v\left(  t,x\right)
v\left(  t,x\right)  \dd x \dd s \ . \label{eq3}
\end{align}

\textbf{Step 3} (final estimates)
We want to find suitable bounds on the last term of (\ref{eq3}) allowing to apply Gronwall inequality. This will complete the proof. 

For every $t\in[0,T]$, we have
\[
\left\vert \int v\, b\cdot\nabla v \, \dd x \right\vert \leq\left(  \int\left\vert \nabla
v\right\vert ^{2} \dd x \right)  ^{1/2}\left(  \int v^{2}\left\vert b\right\vert
^{2} \dd x \right)  ^{1/2};
\]%
\[
\int v^{2}\left\vert b\right\vert ^{2} \dd x \leq\left(  \int v^{2r} \dd x \right)
^{1/r}\left(  \int\left\vert b\right\vert ^{p} \dd x \right)  ^{2/p},
\]
where $1/r+2/p=1$ namely $1/r=1-2/p=(p-2)/p$ :
\[
r=\frac{p}{p-2}\, .
\]
One has the interpolation inequality%
\begin{align*}
\left(  \int v^{\alpha} \dd x \right)  ^{1/\alpha}   \leq\left(  \int
v^{2} \dd x \right)  ^{1-s}\left(  \int\left\vert \nabla v\right\vert
^{2} \dd x \right)  ^{s} , \qquad  \qquad
s   =\frac{\alpha-2}{2\alpha}d \, .
\end{align*}

The idea of the result comes from: $W^{s,2}\subset L^{\alpha}$ for $1/\alpha=1/2-s/d$, namely $\frac{s}{d}=\frac{1}{2}-\frac
{1}{\alpha}=\frac{\alpha-2}{2\alpha}$, $s=\frac{\alpha-2}{2\alpha}d$; and
then
\[
\left(  \int v^{\alpha} \dd x \right)  ^{1/\alpha}\leq\left\Vert v\right\Vert
_{W^{s,2}} \leq  \left\Vert v\right\Vert _{L^{2}}^{1-s}\left\Vert v\right\Vert
_{W^{1,2}}^{s}.
\]
Let us put everything together:%
\begin{align*}
\left(  \int v^{2r}  \dd x \right)  ^{1/r}  & =\left(  \int v^{\alpha}  \dd x \right)
^{2/\alpha}\leq\left(  \int v^{2} \dd x \right)  ^{1-s}\left(  \int\left\vert
\nabla v\right\vert ^{2}  \dd x \right)  ^{s}\\
r  & =\frac{p}{p-2} \, ,  \quad  \alpha=2r \, ,  \quad   s=\frac{\alpha-2}{2\alpha}d
\end{align*}
namely%
\[
s=\frac{2r-2}{4r}d=\frac{2\frac{p}{p-2}-2}{4\frac{p}{p-2}}d=\frac{d}{p} \ .
\]
Thus we have proved:%
\[
\int v^{2}\left\vert b\right\vert ^{2}  \dd x \leq\left(  \int v^{2}  \dd x \right)
^{1-\frac{d}{p}}\left(  \int\left\vert \nabla v\right\vert ^{2}  \dd  x \right)
^{\frac{d}{p}}\left(  \int\left\vert b\right\vert ^{p}  \dd x \right)  ^{2/p}%
\]
and we can bound the last term in (\ref{eq3})

\[
\left\vert \int v\, b\cdot\nabla v\,  \dd x \right\vert \leq\left(  \int\left\vert \nabla
v\right\vert ^{2}  \dd x \right)  ^{\frac{p+d}{2p}}\left(  \int v^{2}  \dd x \right)
^{\frac{p-d}{2p}}\left(  \int\left\vert b\right\vert ^{p}  \dd x \right)  ^{1/p}.
\]
Recall that $ab\leq\frac{a^{s}}{s}+\frac{b^{r}}{r}$, $\frac{1}{s}+\frac{1}{r}=1$.
Then, 
with $s=\frac{2p}{p+d}$, $r=\frac{2p}{p-d}$ we have%
\[
\left\vert \int v\, b\cdot\nabla v\,  \dd x \right\vert \leq\frac{\sigma^{2}}{2}\left(
\int\left\vert \nabla v\right\vert ^{2}  \dd x \right)  +C\left(  \int
v^{2}  \dd x \right)  \left(  \int\left\vert b\right\vert ^{p}  \dd x \right)  ^{\frac
{2}{p-d}}.
\]
Therefore%
\begin{align*}
& \int v^{2}\left(  t,x\right)  \dd x \leq C\int_{0}^{t}\left(  \int v^{2}  \dd x \right)  \left(  \int\left\vert
b\right\vert ^{p}  \dd x \right)  ^{\frac{2}{p-d}}  \dd s
\end{align*}
hence we may apply Gronwall lemma and deduce $\int v^{2}\left(  t,x\right)
\dd x =0$ if%
\[
\int_{0}^{T}\left(  \int\left\vert b\right\vert ^{p}  \dd x \right)  ^{\frac{2}%
{p-d}}  \dd s  <\infty.
\]
We know that
\[
\int_{0}^{T}\left(  \int\left\vert b\right\vert ^{p}  \dd x \right)  ^{\frac{q}{p}%
}  \dd s  <\infty
\]
for certain $p,q\geq2$ such that $\frac{d}{p}+\frac{2}{q}<1$. Then%
\[
\frac{2}{p-d}  <  \frac{q}{p}%
\]
because $2p<qp-qd$, $\frac{2}{q}<1-\frac{d}{p}$. The proof is complete. \ \ \ 
\end{proof}

\section{Appendix: Technical Lemmas}\label{sec app1}

For completeness, we collect here some modifications of known results used in Section \ref{sec 2}. We will use the notation introduced there.

\begin{lemma}\label{lemma U^n}  Let $U_n$ be the solution of the PDE (\ref{PDE}) for $f=b=b^n$, as defined in Section \ref{sec 2}. Then
\begin{itemize}
\item[i)] $U^{n}\left(  t,x\right)$ and $\nabla U^{n}\left(  t,x\right)$ converge pointwise in $\left(t,x\right)$ to $U\left(  t,x\right)$ and $\nabla U\left(  t,x\right)$ respectively, and the convergence is uniform on compact sets;
\item[ii)] there exists a $\lambda$ for which $\sup_{t,x}\left\vert \nabla U^{n}\left(  t,x\right)  \right\vert \le 1/2$;
\item[iii)] $\left\Vert \nabla^{2}U^{n}\left(  t,x\right)  \right\Vert _{L^{q}_{p}\left(T\right)  }   \leq C $ .
\end{itemize}
\end{lemma}

\begin{proof} 
The result of the second point is proved in \cite[Lemma 3.4]{FeFl2} for a fixed $n$, but inspecting the proof we see that all the bounds obtained depend on $\|b\|_{L_p^q}$, but never on $b$ itself. Since $\|b^n\|_{L_p^q} \to \|b\|_{L_p^q}$, the uniformity in $n$ follows.\\
To prove the other two points, set $V^n:= U^n-U$; then
\begin{equation*}
\partial_t V^n + \frac{1}{2}\Delta V^{n} + b\cdot\nabla V^{n} - \lambda V^n = - \big(b^{n} - b \big) \cdot \big( Id + \nabla U^n\big)\, ,\quad V^{n}\left(  T,x\right)  =0\, .
\end{equation*}
From the bound (\ref{Stima di Schauder}) on the solution provided by Theorem \ref{Main PDE Theorem}, we obtain
\begin{equation*}
\big\| V^n \big\|_{H_{2,p}^q} \le  N \|b^n-b\|_{L_p^q} \to 0 \ .
\end{equation*}
It follows that $U^n\to U$ in $H_{2,p}^q$, which proves the last point. Since by  \cite[Lemma 10.2]{KR} $U, U^n, \nabla U$ and $\nabla U^n$ are all H\"older continuous functions, there exists a subsequence (that we still call $U^n$) s.t. $U^n\to U$ and $\nabla U^n\to \nabla U$ for every $(t,x)$ and uniformly on compacts. \ \ \ \end{proof}

\begin{remark}\label{uniformity in n} The following results hold uniformly in $n$ because, as remarked in the previous proof, all the bounds obtained depend on the norm of $b$.
\begin{itemize}
\item[i)] From \cite[Lemma 3.5]{FeFl2} we have
\begin{equation} \label{stima invertibilita}
\sup_n\sup_{t\in[0,T]}\big| \nabla ( \gamma^n_t)^{-1} (\cdot) \big|_{\calC(\R^d)} \le 2 \ ;
\end{equation}
\item[ii)] from the uniform boundedness of the coefficients ($U^n$ and $\nabla U^n$) of the SDE (\ref{new SDE}), we get
\begin{equation}\label{integrability Y}
\sup_{t\in[0,T]} \E \Big[ \big| \psi_t^n(x)  \big|^a  \Big] \le C \Big( 1+ |x|^a \Big)\ .
\end{equation}
\end{itemize}
\end{remark}

\begin{lemma}\label{lem e^A^n}
For every $n$, both the process $A^n$ defined by (\ref{eq A^n}) and the one defined by (\ref{new A^n}) are continuous, adapted, nondecreasing, with $A_0^n=0$, $\E [ A_T^n] \le C$ and for every $k\in \R$, $\E [ e^{k A_T^n}]\le C$. The constant $C$ can be chosen independently of $n$.
\end{lemma}

\begin{proof}
For the process defined by (\ref{eq A^n}) the proof follows the same steps of the proof of \cite[Lemma 7]{FeFl1}. We only remark that the function $U$ is the solution of a different PDE, but it has the same properties in terms of regularity. Moreover, the flows $\phi$ and $\phi^n$ solve two SDEs with different drifts $b$ and $b^n$, which means that in the proof one has to use twice the result of \cite[Corollary 13]{FeFl1}, once for every drift.

For the process defined by (\ref{new A^n}), the result is already contained in \cite[Corollary 13]{FeFl1}. \ \ \ 
\end{proof}

\begin{lemma}\label{lemma integrability b(X)}
Let $f^n$ be a sequence of vector fields belonging to $L_p^q$, convergent to $f\in L_p^q$. Then, there exists $\varepsilon>1$ s.t.
\begin{equation}\label{eq integrability b(X)}
\E \Big[ \int_0^T  \big| f^n(s,\phi_s^n)  \big|^{2\varepsilon}  \dd s  \Big] \le C <\infty \ .
\end{equation} 
\end{lemma}

\begin{proof}   
To prove the result for a fixed $n$ one can use \cite[Corollary 13]{FeFl1} and follow the proof of \cite[Lemma 8 and Corollary 9]{FeFl1}, which still works due to the strict inequality in the conditions imposed on $p,q$. Then, since all the bounds only depend on the norm of $f$ but never on the function itself, one obtains that (\ref{eq integrability b(X)}) is uniform in $n$. \ \ \ 
\end{proof}

\section{Appendix: Sobolev Regularity of Random Fields}

Let $r\geq1$ be given. We \ recall that $f\in W_{loc}^{1,r}\left(
\mathbb{R}^{d}\right)  $ if $f\in L_{loc}^{r}\left(  \mathbb{R}^{d}\right)  $
and there exist $g_{i}\in L_{loc}^{r}\left(  \mathbb{R}^{d}\right)  $,
$i=1,...,d$, such that
\[
\int_{\mathbb{R}^{d}}f\left(  x\right)  \partial_{x_{i}}\varphi\left(
x\right)  dx=-\int_{\mathbb{R}^{d}}g_{i}\left(  x\right)  \varphi\left(
x\right)  dx
\]
for all $\varphi\in C_{0}^{\infty}\left(  \mathbb{R}^{d}\right)  $. When this
happens, we set $\partial_{x_{i}}f\left(  x\right)  =g_{i}\left(  x\right)  $.
From the definition and easy arguments one has the following criterion: if
$f\in L_{loc}^{r}\left(  \mathbb{R}^{d}\right)  $ and there exist a sequence
$\left\{  f_{n}\right\}  \subset W_{loc}^{1,r}\left(  \mathbb{R}^{d}\right)  $
such that $f_{n}\rightarrow f$ in $L_{loc}^{1}\left(  \mathbb{R}^{d}\right)  $
(or even in distributions) and for all $R>0$ one has a constant $C_R  >0$ such that $\int_{B_R  }\left\vert \nabla
f_{n}\left(  x\right)  \right\vert ^{r}dx\leq C_R $ uniformly
in $n$, then $f\in W_{loc}^{1,r}\left(  \mathbb{R}^{d}\right)  $. This
criterion will not be used below; it is only stated for comparison with the
next result.

Let now $F:\Omega\times\mathbb{R}^{d}\rightarrow\mathbb{R}$ be a random field.
When we use below this name we always assume it is jointly measurable.

\begin{lemma}
\label{lemma appendix} Assume that $F\left(  \omega,\cdot\right)  \in
L_{loc}^{r}\left(  \mathbb{R}^{d}\right)  $ for $P$-a.e. $\omega$ and there
exist a sequence $\left\{  F_{n}\right\}  _{n\in\mathbb{N}}$ of random fields
such that

\begin{enumerate}
\item $F_{n}\left(  \omega,\cdot\right)  \rightarrow F\left(  \omega
,\cdot\right)  $ in distributions in probability, namely%
\[
P-\lim_{n\rightarrow\infty}\int_{\mathbb{R}^{d}}F_{n}\left(  \omega,x\right)
\varphi\left(  x\right)  \dd x   =   \int_{\mathbb{R}^{d}}F\left(  \omega,x\right)
\varphi\left(  x\right)  \dd x
\]
for every $\varphi\in C_{0}^{\infty}\left(  \mathbb{R}^{d}\right)  ;$

\item $F_{n}\left(  \omega,\cdot\right)  \in W_{loc}^{1,r}\left(
\mathbb{R}^{d}\right)  $ for $P$-a.e. $\omega$ and for every $R>0$ 
there exists a constant $C_R>0$ such that
\[
\E \left[  \int_{B_R  }\left\vert \nabla F_{n}\left(  x\right)
\right\vert ^{r}  \dd x \right]  \leq C_R
\]
uniformly in $n$.
\end{enumerate}

Then $F\left(  \omega,\cdot\right)  \in W_{loc}^{1,r}\left(  \mathbb{R}%
^{d}\right)  $ for $P$-a.e. $\omega$,%
\begin{equation}
\E \left[  \int_{\mathbb{R}^{d}}\partial_{x_{i}}F\left(  \cdot,x\right)
\varphi\left(  x\right)  Z \,  \dd x  \right]  =  \lim_{n\rightarrow\infty}  \E \left[
\int_{\mathbb{R}^{d}} \partial_{x_{i}} F_{n} \left(  \cdot,x\right)
\varphi\left(  x\right)  Z \,  \dd x \right]  \label{convergence 2}%
\end{equation}
for all $\varphi\in C_{0}^{\infty}\left(  \mathbb{R}^{d}\right)  $ and bounded
r.v. $Z$,%
\begin{equation}
P-\lim_{n\rightarrow\infty}\int_{\mathbb{R}^{d}}\partial_{x_{i}}F_{n}\left(
\omega,x\right)  \varphi\left(  x\right)  \dd x  =  - \int_{\mathbb{R}^{d}}F\left(
\omega,x\right)  \partial_{x_{i}}\varphi\left(  x\right)
\dd x \label{convergence 3}%
\end{equation}
for all $\varphi\in C_{0}^{\infty}\left(  \mathbb{R}^{d}\right)  $, and 
for every $R>0$
\begin{equation}
\E\left[  \int_{B_R }\left\vert \nabla F\left(  x\right)
\right\vert ^{r}  \dd x \right]  \leq\underset{n\rightarrow\infty}{\lim\sup} \ \E\left[
\int_{B_R  }\left\vert \nabla F_{n}\left(  x\right)
\right\vert ^{r}  \dd x \right]  .\label{property 3}%
\end{equation}

\end{lemma}

\begin{proof}
Given $R>0$,  there is a subsequence $\left\{  n_{k}\right\}  $ and a
vector valued random field $G$ such that $\nabla F_{n_{k}}$ converges
weakly to $G$ in $L^{r}\left(  \Omega\times \R^d  \right)$, as $k\rightarrow\infty$. Taking $R\in\mathbb{N}$, we may apply a diagonal procedure and find a single subsequence $\left\{  n_{k}\right\}  $ and vector
valued random fields $G^{R}$, $R\in\mathbb{N}$, such that $\nabla F_{n_{k}}$
converges weakly to $G^{R}$ in $L^{r}\left(  \Omega\times B_R
\right)  $, as $k\rightarrow\infty$, for each $R\in\mathbb{N}$. Using suitable
test functions, one can see that $G^{R^{\prime}}=G^{R}$ on $\Omega\times
B_R$ if $R^{\prime}>R$. Hence we have found a single vector
valued random field $G$, such that $\nabla F_{n_{k}}$ converges weakly to $G$
in $L^{r}\left(  \Omega\times B_R  \right)  $, as
$k\rightarrow\infty$, for each $R\in\mathbb{N}$ and thus for each real $R>0$.
At the end of the proof, $G$ will be identified by $\nabla F$, independently of the
subsequence $\left\{  n_{k}\right\}  $. Thus, a fortiori, the full sequence
$\nabla F_{n}$ converges weakly to $G$ in $L^{r}\left(  \Omega\times \R^d \right) $, as $n\rightarrow\infty$. For this
reason, to simplify notations, we omit the notation of the subsequence. 

For each $\varphi\in C_{0}^{\infty}( \mathbb{R}^{d})  $, by
assumptions 1 and 2
\begin{align*}
\int_{\mathbb{R}^{d}}F\left(  \omega,x\right)  \partial_{x_{i}}\varphi\left(
x\right)  \dd x  & =  \lim_{n\rightarrow\infty}\int_{\mathbb{R}^{d}}F_{n}\left(
\omega,x\right)  \partial_{x_{i}}\varphi\left(  x\right)  \dd x \\
& = - \lim_{n\rightarrow\infty}\int_{\mathbb{R}^{d}}\partial_{x_{i}}F_{n}\left(
\omega,x\right)  \varphi\left(  x\right)  \dd x \, ,
\end{align*}
the limits being understood in probability. For each bounded r.v. $Z$, this
implies that
\[
\lim_{n\rightarrow\infty}\int_{\mathbb{R}^{d}}\partial_{x_{i}}F_{n}\left(
\omega,x\right)  \varphi\left(  x\right)  Z\left(  \omega\right)
\dd x   = -  \int_{\mathbb{R}^{d}}F\left(  \omega,x\right)  \partial_{x_{i}}%
\varphi\left(  x\right)  Z\left(  \omega\right)  \dd x
\]
in probability. This limit holds also in $L^{1}\left(  \Omega\right)  $ by
Vitali convergence criterion because, by H\"older inequality,
\begin{align*}
&\E \left[  \left\vert \int_{\mathbb{R}^{d}}  \partial_{x_{i}}F_{n}\left(
\omega,x\right)  \varphi\left(  x\right)  Z\left(  \omega\right)
\dd x \right\vert ^{p}\right] \\
& \hspace{4cm} \leq C_{R,p} \left\Vert \varphi\right\Vert _{L^{\infty}} \| Z\|_{L^\infty} \E \left[  \int_{B_R  } \left\vert \partial_{x_{i}}
F_{n}\left(  \omega,x\right)  \right\vert ^{p}  \dd x \right]  \\
& \hspace{4cm}  \leq C_{R,p} \left\Vert \varphi\right\Vert _{L^{\infty}} \| Z\|_{L^\infty} C_R
\end{align*}
uniformly in $n$, for some $p>1$, and with $R$ such that $B_R$ contains the support of $\varphi$.

From the weak convergence above, we also get that 
\[
\lim_{n\rightarrow\infty}\E\left[  \int_{\mathbb{R}^{d}}\partial_{x_{i}}%
F_{n}\left(  \cdot,x\right)  \varphi\left(  x\right)  Z \dd x \right]  = \E \left[
\int_{\mathbb{R}^{d}}G_{i}\left(  \cdot,x\right)  \varphi\left(  x\right)
Z \dd x \right]  .
\]
Hence%
\[
\E\left[  \int_{\mathbb{R}^{d}}G_{i}\left(  \cdot,x\right)  \varphi\left(
x\right)  Z \dd x\right]  = - \E\left[  \int_{\mathbb{R}^{d}}F\left(  \cdot,x\right)
\partial_{x_{i}}\varphi\left(  x\right)  Z \dd x\right]  .
\]
By the arbitrariness of $Z$ this gives us%
\begin{equation}
\int_{\mathbb{R}^{d}}F\left(  \omega,x\right)  \partial_{x_{i}}\varphi\left(
x\right)  \dd x  =  - \int_{\mathbb{R}^{d}}G_{i}\left(  \omega,x\right)
\varphi\left(  x\right)  \dd x \label{identity 1}%
\end{equation}
for $P$-a.e. $\omega$. This is the identification of $G$ mentioned above,
which implies the weak convergence of the full sequence $\nabla F_{n}$. 

Identity (\ref{identity 1}) holds $P$-a.s. for every $\varphi\in C_{0}%
^{\infty}(  \mathbb{R}^{d})  $ a priori given. Thus it holds
$P$-a.s., uniformly on a dense countable set $\mathcal{D}$ of test functions
$\varphi$, dense for instance in $W_{loc}^{1,r^{\prime}}(  \mathbb{R}%
^{d})  $, $\frac{1}{r}+\frac{1}{r^{\prime}}=1$. Using the integrability
properties of $F\left(  \omega,\cdot\right)  $ and $G_{i}\left(  \omega
,\cdot\right)  $ we may extend identity (\ref{identity 1}) to all $\varphi\in
W_{loc}^{1,r^{\prime}}(  \mathbb{R}^{d})  $ and thus all
$\varphi\in C_{0}^{\infty}(  \mathbb{R}^{d})  $, uniformly with
respect to the good set of $\omega$ for which it holds on $\mathcal{D}$. 

Thus, from identity (\ref{identity 1}) in the stronger form just explained, we
deduce - by definition - that $F\left(  \omega,\cdot\right)  \in W_{loc}%
^{1,r}(  \mathbb{R}^{d})  $ for $P$-a.e. $\omega$. And $\nabla
F\left(  \omega,x\right)  =G\left(  \omega,x\right)  $. We immediately have
(\ref{convergence 2}) and (\ref{convergence 3}). 

We have shown that, for every function $\xi(\omega,x)$ of the form
\begin{equation*}
\xi(\omega,x) = \sum_{k=1}^m \varphi_k(x) Z_k (\omega)\, ,
\end{equation*}
 with $\varphi$ and $Z$ as above, we have
\begin{align*}
\left\vert \E \left[  \int_{B_R}\nabla F\left(  \cdot,x\right) \xi(\cdot,x) \,  \dd x \right]  \right\vert  
& \leq\underset{n\rightarrow \infty}{\lim}\left\vert \E \left[  \int_{B_R}\nabla F_{n}\left(
\cdot,x\right)  \xi(\cdot,x) \,    \dd x \right]  \right\vert \\
& \hspace{-3cm} \leq \E \left[  \int_{B_R }\left\vert \xi(\cdot,x) \right\vert ^{r^{\prime}}  \dd x \right]  ^{1/r^{\prime}}  \underset{n\rightarrow\infty}{\lim\sup} \, \E \left[  \int_{B_R } \left\vert \nabla F_{n}\left( \cdot, x\right)  \right\vert
^{r} \dd x\right]  ^{1/r} %
\end{align*}
(in the last passage we have used H\"older inequality). The set of functions $\xi$ introduced is dense in $L^{r'} (\Omega\times B_R ) $, so that
\begin{align*}
 \E \bigg[ \int_{B_R} \big| \nabla F(\cdot,x) \big|^r \dd x \bigg]^{1/r} & = \big\| \nabla F \big \|_{L^r(\Omega \times B_R)} \\
 &\le \sup_{\|\xi\|_{L^{r'}} \le 1} \left|  \E \left[  \int_{B_R}\nabla F\left(  \cdot,x\right) \xi(\cdot,x) \,  \dd x \right]  \right| \\
 &\le \underset{n\rightarrow\infty}{\lim\sup} \, \E \left[  \int_{B_R } \left\vert \nabla F_{n}\left( \cdot, x\right)  \right\vert
^{r} \dd x\right]  ^{1/r} .
\end{align*}  
This completes the proof. \ \ \ 
\end{proof}


\clearpage

\begin{thebibliography}{9}

\bibitem{Amb} L. Ambrosio, \textit{Transport equation and Cauchy problem for BV vector fields},
Invent. Math. 158 (2004) 227-260.

\bibitem{Att} S. Attanasio, \textit{Stochastic flows of diffeomorphisms for one-dimensional SDE with discontinuous drift}, Electron. Commun. Probab. 15 (2010) 213--226.

\bibitem{BD1} A. de Bouard, A. Debussche, \textit{On the effect of the noise on the solutions of
supercritical Schršdinger equation}, Probab. Theory Related Fields 123 (2002) 76--96.

\bibitem{BD2} A. de Bouard, A. Debussche, \textit{Finite time blow-up in the additive
supercritical stochastic nonlinear Schršdinger equation: the real noise case}, Contemporary Math. 301 (2002) 183--194.

\bibitem{BD3} A. de Bouard, A. Debussche, \textit{Blow-up for the stochastic nonlinear
Schršdinger equation with multiplicative noise}, Annals of Probab. 33 (3) (2005) 1078--1110.

\bibitem{BD4} A. de Bouard, A. Debussche, \textit{The nonlinear Schr\"odinger equation with white
noise dispersion}, J. Funct. Anal. 259 (5) (2010) 1300--1321.

\bibitem{DM1} A. Debussche, L. Di Menza, \textit{Numerical simulation of focusing stochastic
nonlinear Schr\"odinger equations}, Physica D 162 (2002) 131--154.

\bibitem{DM2} A. Debussche, L. Di Menza, \textit{Numerical resolution of stochastic focusing NLS
equations}, Appl. Math. Letters 15 (6) (2002)  661--669.

\bibitem{DT} A. Debussche, Y. Tsutsumi, \textit{1D quintic nonlinear Schr\"odinger equation with
white noise dispersion}, J. Math. Pures Appl. 96 (4) (2011) 363--376.

\bibitem{DFV} F. Delarue, F. Flandoli, D. Vincenzi, \textit{Noise prevents collaps of Vlasov-Poisson point charges},
 preprint.

\bibitem{DPL} R. J. DiPerna, P. L. Lions, \textit{Ordinary differential equations, transport
theory and Sobolev spaces}, Invent. Math. 98  (1989) 511-547.

\bibitem {FeFl1} E. Fedrizzi, F. Flandoli, \textit{Pathwise uniqueness and continuous dependence for SDEs with nonregular drift}, Stochastics 83 (3) (2011) 241--257.
                                                                          
\bibitem {FeFl2} E. Fedrizzi, F. Flandoli, \textit{H\"older Flow and Differentiability for SDEs with Nonregular Drift}, to appear in Stochastic Analysis and Applications (2012).

\bibitem{F} F. Flandoli, \textit{Random Perturbation of PDEs and Fluid Dynamic Models}, Saint Flour
summer school lectures 2010, Lecture Notes in Mathematics n. 2015, Springer, Berlin (2011).

\bibitem {FGP1} F. Flandoli, M. Gubinelli and E. Priola, \textit{Well--\thinspace posedness of the transport equation by stochastic perturbation}, Invent. Math. 180 (1) (2010), 1--53.

\bibitem{FGP2} F. Flandoli, M. Gubinelli, E. Priola, \textit{Full well-posedness of point vortex dynamics corresponding to stochastic 2D Euler equations}, Stoch. Proc. Appl., 121 (7) (2011) 1445--1463.

\bibitem {KR} N.V. Krylov and M. R\"{o}ckner, \textit{Strong solutions to stochastic equations with singular time dependent drift}, Probab. Theory Relat. Fields 131 (2005) 154--196.

\bibitem{Ku} H. Kunita, \textit{Stochastic flows and stochastic differential equations}, Cambridge studies in advanced mathematics, Cambridge university press (1990).

\bibitem{MP} T. Meyer-Brandis and F. Proske, \textit{Construction of strong solutions of SDE's via Malliavin calculus}, J. Funct. Anal. 258 (11) (2010) 3922--3953.

\bibitem{MNP} S.E.A. Mohammed, T.K. Nilssen, F.N. Proske, \textit{Sobolev Differentiable Stochastic Flows of SDE's with Measurable Drift and Applications}, preprint, arXiv:1204.3867.

\bibitem {Z} X. Zhang, \textit{Stochastic Homeomorphism Flows of SDEs with Singular Drifts and Sobolev Diffusion Coefficients},  Electronic Journal of Probability 16 (38) (2011) 1096--1116.

\end{thebibliography}
\end{document}